\documentclass[12pt,reqno]{amsart}
\usepackage{amssymb}
\usepackage{mathrsfs}\overfullrule=5pt
\usepackage{amsmath,amssymb,graphicx}
\usepackage[colorlinks=true,urlcolor=red, citecolor=blue]{hyperref}

\usepackage{xcolor}
\newtheorem{thm}{Theorem}[section]
\newtheorem{lem}[thm]{Lemma}

\newtheorem{prop}[thm]{Proposition}

\newtheorem{cor}[thm]{Corollary}

\theoremstyle{definition}

\theoremstyle{remark}
\newtheorem{rem}[thm]{Remark}

\numberwithin{equation}{section}

\newcommand{\bbz}{\mathbb{Z}}
\newcommand{\bbn}{\mathbb{N}}

\begin{document}
\title
[$\Delta$-weakly mixing subsets along a collection of sequences]
{$\Delta$-weakly mixing subsets along a collection of sequences of integers}
\author[J. Li] {Jian Li}
\address[J. Li]{
Department of Mathematics, Shantou University, Shantou, Guangdong  515063, People's Republic of China}
\email{lijian09@mail.ustc.edu.cn}
\author[K. Liu]{Kairan Liu}
\address[K. Liu]{
Department of Mathematics, University of Science and Technology of China, Hefei, Anhui 230026, People's Republic of China}
\email{lkr111@mail.ustc.edu.cn}

\date{\today}

\begin{abstract}
In this paper, we propose a mild condition, named Condition $(**)$,
for collections of sequence of integers and show that
for any measure preserving system
the Pinsker $\sigma$-algebra is a characteristic $\sigma$-algebra
for the averages along a collection satisfying Condition $(**)$.
We introduce the notion of $\Delta$-weakly mixing subsets
along a collection of sequences of integers
and show that positive topological entropy implies
the existence of $\Delta$-weakly mixing subsets
along a collection of ``good'' sequences.
As a consequence, we show that
positive topological entropy implies
multi-variant Li-Yorke chaos along polynomial times of the shift prime numbers.
\end{abstract}

\maketitle

\section{Introduction}
Throughout this paper, by a \textit{topological dynamical system}, we mean a pair $(X, T)$, where $X$ is a compact metric space and $T\colon X \to X$ is a homeomorphism, and by a \textit{measure preserving system}, we mean a quadruple $(X,\mathscr{B},\mu,T)$,
where $(X,\mathscr{B})$ is a standard Borel space, $\mu$ is a probability measure on $(X,\mathscr{B})$
and $T\colon (X,\mathscr{B},\mu) \to (X,\mathscr{B},\mu)$ is an invertible measure preserving transformation.

Chaos, as an important concept representing complexity of topological dynamical system, has attracted a lot of attention. Different versions of chaos, such as Li-Yorke chaos, Devaney chaos, positive entropy and weak mixing were proposed over the past few decades, and the implication among them became a central topic as well. See a recent survey \cite{LY16} and references therein for more details. Here we name a few related to this work.

In \cite{Iwanik}, Iwanik proved that weak mixing implies Li-Yorke chaos. By showing that a non-periodic transitive system with a periodic point is Li-Yorke chaotic, Huang and Ye proved that Devaney chaos also implies Li-Yorke one (see \cite{Huang-Ye1}).
In \cite{Blanchard-Glansner-Kolyada-Maass}, Blanchard et al. proved that positive topological entropy implies Li-Yorke chaos,
see \cite{Kerr-li}, \cite{Huang-Xu-Yi} and \cite{Wang-Zhang} for amenable group actions. In \cite{D14}, Downarowicz proved that positive topological entropy implies mean Li-Yorke chaos, see also \cite{HLY14} for another approach.

In \cite{Xiong-Yang}, Xiong and Yang showed that in a weakly mixing system there are considerably many points in the domain whose orbits display highly erratic time dependence, which is called Xiong chaos. In \cite{Blanchard-Huang}, Blanchard and Huang defined a local version of weak mixing, so called weakly mixing set, and proved that positive topological entropy implies the existence of weakly mixing sets which also implies Li-Yorke chaos. In \cite{Piotr-Zhang(1),Piotr-Zhang(2),Piotr-Zhang(3)}, Oprocha and Zhang had also discussed local versions of weak mixing extensively.  In \cite{HLYZ17}, the first named author of the present paper with coauthors studied the $\Delta$-weakly mixing  property and showed that a topological dynamical system with positive topological entropy has many $\Delta$-weakly mixing subsets. Recently in \cite{L19}, the second named author of the present paper extended the results in \cite{HLYZ17} to countable torsion-free discrete nilpotent group actions.

Recently, there are some papers which studied the chaos phenomenons along some subsequences of integers.  In fact, any system containing a non-trivial weakly mixing subset has infinite topological sequence entropy, see \cite[Theorem 6.1]{Piotr-Zhang(2)}. In \cite{L15}, the first named author studied weakly mixing subsets via Furstenberg families (collections of subsequence of integers). In \cite{Li-Qiao}, the first named author with Qiao showed that positive topological entropy implies mean Li-Yorke chaos along some sequences which are good for pointwise ergodic convergence with a mild condition. In \cite{HLY18}, the first named author with coauthors showed that positive topological entropy implies Li-Yorke chaos along  any infinite sequence for countable discrete amenable group actions. In \cite{ZWX19}, Zhang et al. studied $\Delta$-mixing properties via Furstenberg families. The main aim of this paper is to study $\Delta$-weakly mixing set along a collection of sequences of integers and show that positive topological entropy implies the existence of $\Delta$-weakly mixing sets along
a collection of ``good'' sequences.

To introduce our main results, we need some preparations. Let $\bbn$ denote the collection of all positive integers. For a topological dynamical system $(X,T)$ and $m\in\bbn$, the $m$-th product system $(X^m,T^{(m)})$ is also a topological dynamical system, where $X^m=X\times X\times \dotsb\times X$, and $T^{(m)}(x_1,x_2,\dotsc,x_m)=(Tx_1,Tx_2,\dotsc,Tx_m)$ for any $(x_1,x_2,\dotsc,x_m)\in X^m$. We say that a topological dynamical system $(X, T)$ is \emph{(topologically) transitive} if for every two non-empty open subsets $U$ and $V$ of $X$ there exists a positive integer $n$ such that $U\cap T^{-n}V\neq\emptyset$.
It is \emph{(topologically) weakly mixing} if the product system $(X^2, T^{(2)})$ is transitive. By the well-known Furstenberg intersection Lemma, we know that if $(X, T)$ is weakly mixing then $(X^m,T^{(m)})$ is also transitive for all $m\in\bbn$.

Following \cite{M10}, we say that a topological dynamical system $(X, T)$ is \emph{$\Delta$-transitive} if for every $\ell\in\bbn$,
there exists a residual subset $X_0$ of $X$ such that for every $x\in X$,
\[\lbrace (T^nx,T^{2n}x,\dotsc, T^{\ell n} x)\colon n\in\bbn\rbrace \]
is dense in $X^\ell$. In \cite{G94}, Glasner showed that if a minimal system is weakly mixing then it is $\Delta$-transitive.
Recently in \cite{HSY16}, Huang, Shao and Ye showed that if a minimal system is weakly mixing then for every $\ell\in\bbn$ and distinct  non-constant polynomials $p_1(n),p_2(n),\dotsc,p_\ell(n)$ with rational coefficients taking integer values on the integers and $p_i(0)=0$ for $i=1,\dotsc,\ell$, there exists a residual subset $X_0$ of $X$ such that for every $x\in X$,
\[\lbrace (T^{p_1(n)}x,T^{p_2(n)}x,\dotsc, T^{p_l(n)} x)\colon n\in\bbz \rbrace\]
is dense in $X^\ell$. In fact, they proved the result for finitely generated nilpotent group actions, see \cite{HSY16} for more details. This motivates us to introduce $\Delta$-transitive and $\Delta$-weakly mixing set along a collection of sequences of integers.

Let $\ell\in\bbn$ and $\Lambda=\{a_1,\dotsc,a_\ell\colon\bbn\to\bbn\}$ be a collection of sequences. Assume that $(X,T)$ is a topological dynamical system and $E$ is a closed subset of $X$.
We say that $E$ is \emph{$\Delta$-transitive along $\Lambda$}
if there exists a residual subset $A$ of $E$ such that for every $x\in A$ one has
\[ E^\ell  \subseteq \overline{\{ (T^{a_1(n)}x,\dotsc,T^{a_\ell(n)}x)\colon n\in\bbn\}}, \]
and \emph{$\Delta$-weakly mixing along $\Lambda$} if $E^m$ is a $\Delta$-transitive subset of $(X^m,T^{(m)})$ along $\Lambda$ for every $m\in\bbn$.

To study the existence of $\Delta$-weakly mixing sets in dynamical system with positive topological entropy, more things should be involved. Given a measure preserving system  $(X,\mathscr{B},\mu,T)$, we say that a sub-$\sigma$-algebra $\mathscr{F}$ of $\mathscr{B}$ is a \emph{characteristic factor} for the averages along the collection $\Lambda=\{a_1,\dotsc,a_\ell\}$ of sequences if for any $f_1,\dotsc,f_\ell \in L^\infty(\mu)$, one has
\[
\lim\limits_{N\to\infty}\frac{1}{N}\sum_{n=1}^{N}\bigg(\prod_{i=1}^{\ell}T^{a_i(n)}f_i-\prod_{i=1}^{\ell}T^{a_i(n)}
 E(f_i|\mathscr{F})\bigg)=0
 \]
in $L^2(\mu)$.

In \cite{Li-Qiao}, the first named author with Qiao introduced the following Condition $(*)$ for a sequence $a(n)$ of positive integers.

\medskip
\noindent\textbf{Condition ($*$).}
For every $L>0$,
\[\lim_{N\to\infty}\frac{1}{N^2}\#\{(m,n)\in[1,N]^2:\,|a(m)-a(n)|\leq L\}=0.\]
It is shown in \cite[Theorem 3.2]{Li-Qiao} that
if a sequence $a(n)$ satisfies the Condition ($*$) then
the Pinsker $\sigma$-algebra of any measure preserving system is a characteristic factor
for the average along the sequence $a(n)$.
Here we introduce the following Condition $(**)$ for a collection $\Lambda=\{a_1,\dotsc,a_\ell\}$ of sequences of positive integers.

\medskip
\noindent\textbf{Condition ($**$).}
For each $i=1,2,\dotsc,\ell$, sequence $a_i(n)$ satisfies the Condition ($*$) and
\[
\lim_{n\to\infty }(a_{i+1}(n)-a_i(n))=+\infty \textrm{ for } i=1,\dotsc,\ell-1,
\]

We have the following improvement of \cite[Theorem 3.2]{Li-Qiao},  which is also of independent  interest.

\begin{thm}\label{main-thm1}
If a collection $\Lambda=\{a_1,\dotsc,a_\ell\}$ of sequences of positive integers  satisfies the Condition ($**$),
then the Pinsker $\sigma$-algebra of any measure preserving system
is a characteristic factor for the average along $\Lambda$.
\end{thm}

Another important ingredient of our method is the following notion.
We say that a collection $\Lambda=\{a_1,\dotsc,a_\ell\}$ of sequences of positive integers  is \emph{good for $\liminf$-$\ell$-recurrence}
if for any measure preserving system $(X,\mathscr{B},\mu,T)$, and every $A\in\mathcal{B}$ with $\mu(A)>0$ one has
\[
\liminf_{N\to\infty}\frac{1}{N}\sum_{n=1}^{N}
\mu(A\cap T^{-a_1(n)}A\cap \dotsb \cap T^{-a_\ell(n)}A)>0.
\]

In fact, by von Neumann ergodic theory, we know that
the sequence $\{n\}_{n=1}^{\infty}$ is good for $\liminf$-$1$-recurrence.
By the seminal work of Furstenberg \cite{F77},
for any $\ell\in\bbn$, the collection $\{n,2n,\dotsc,\ell n\colon n\in \bbn\}$ is good for $\liminf$-$\ell$-recurrence.
In \cite{Bereelson-Leibman},
Bergelson and Leibman showed that for any $\ell\in\bbn$
and any polynomials $p_1(n),\dotsc,p_\ell(n)$
with rational coefficients taking integer values on the integers and
$p_i(0) = 0$ for $i=1,\dotsc,\ell$, the collection $\big\{p_1(n),\dotsc,p_\ell(n)\}$ is good for $\lim\inf$-$\ell$-recurrence.
In \cite{Wooley-Ziegler}, Wooley and Ziegler showed that
for any $\ell\in\bbn$
and any polynomials $p_1(n),\dotsc,p_\ell(n)$
with rational coefficients taking integer values on the integers and
$p_i(0) = 0$ for $i=1,\dotsc,\ell$,
$\big\{p_1(n-1),\dotsc,p_\ell(n-1)\colon n\in\mathbb{P}\}$ and $\big\{p_1(n+1),\dotsc,p_\ell(n+1)\colon n\in\mathbb{P}\}$  are good for $\lim\inf$-$\ell$-recurrence, where $\mathbb{P}$ is the set of prime numbers, see also \cite{FHK13}.
We refer the reader to the survey \cite{F16} and references therein for more results.

Now we are ready to state our main result.

\begin{thm}\label{main-thm2}
Let $\Lambda=\{a_1,\dotsc,a_\ell\}$ be a collection of sequences of positive integers which satisfies the Condition $(**)$
and is good for $\liminf$-$\ell$-recurrence. If a topological dynamical system $(X,T)$ has positive topological entropy,
then there exist $\Delta$-weakly mixing subsets of $(X,T)$ along $\Lambda$.
\end{thm}

In \cite{Piotr-Zhang(3)}, Oprocha and Zhang gave a survey of recent results  on local aspects of dynamics of pairs, tuples and sets, especially on weakly mixing pairs, tuples and sets.
As pointed out by the referee, it is interesting to study corresponding
results on $\Delta$-weakly mixing sets, which is left for further study.

This paper is organized as follows. In Section 2, we review some necessary notions and required properties. In Section 3, we study some properties of $\Delta$-weakly mixing sets along a collection of sequences.
 Theorems \ref{main-thm1} is proved in Sections 4.
Section 5 is devoted to proving the main result Theorem \ref{main-thm2}.

\medskip

\noindent \textbf{Acknowledgments.}
The authors were supported in part by the
NNSF of China (1771264, 11871188, 11801538)
and NSF of Guangdong Province (2018B030306024).
The authors would like to thank Prof. Wen Huang and Prof. Song Shao for their useful comments and suggestions.
The authors would also like to thank the anonymous referee for the careful reading and helpful suggestions.

\section{Preliminaries}
In this section we will review some notions and properties that will be used later, such as density of sets of positive integers, condition expectation and disintegration of measures over sub-$\sigma$-algebras.

\subsection{Density of subsets of positive integers}
Let $F$ be a subset of $\bbn$, the \emph{upper density} and \emph{lower density} of $F$ is defined respectively by
$$\overline{D}(F)=\limsup\limits_{n\to\infty}\frac{\#(F\cap\{1,2,\dotsc,n\})}{n}$$
and
$$\underline{D}(F)=\liminf\limits_{n\to\infty}\frac{\#(F\cap\{1,2,\dotsc,n\})}{n},$$
where $\#(\cdot)$ is the number of elements of a finite set.
We say that $F$ has \emph{density} $D(F)$ if $\overline{D}(F)=\underline{D}(F)$, where $D(F)$ denotes  this common value. It is clear that for two subsets $F$ and $E$ of $\bbn$ with $D(F)=1$ and $\underline{D}(E)>0$, one has $F\cap E\neq\emptyset$.

\subsection{Condition expectation and disintegration of measures}

Let $(X,T)$ be a topological dynamical system.
We denote the collection of all Borel probability measures of $X$ by $\mathcal{M}(X)$, the collection of all $T$-invariant Borel probability measures of $X$ by $\mathcal{M}(X,T)$, and the collection of all ergodic measures of $(X,T)$ by $\mathcal{M}^e(X,T)$.
We now recall the main results and properties of condition expectation and disintegration of measures. We refer to
\cite[Chapter 5]{Einsiedler-Ward} for more details.

Let $(X,\mathscr{B},\mu)$ be a probability space,  and $\mathscr{A}\subseteq\mathscr{B}$ a sub-$\sigma$-algebra. Then there is a map
$$E(\cdot|\mathscr{A}):L^1(X,\mathscr{B},\mu)\to L^1(X,\mathscr{A},\mu)$$
called the \emph{conditional expectation}, that satisfies the following properties.
\begin{enumerate}
\item For $f\in L^1(X,\mathscr{B},\mu)$, the image function $E(f|\mathscr{A})$ is characterized almost everywhere by the two properties:
\begin{itemize}
	\item $E(f|\mathscr{A})$ is $\mathscr{A}$-measurable;
	\item for any $A\in\mathscr{A}$, $\int_A E(f|\mathscr{A})d\mu=\int_A fd\mu$.
\end{itemize}
\item $E(\cdot|\mathscr{A})$ is a linear operator of norm 1. Moreover, $E(\cdot|\mathscr{A})$ is positive.
\item For $f\in L^1(X,\mathscr{B},\mu)$ and $g\in L^{\infty}(X,\mathscr{A},\mu)$,
    $$E(g\cdot f|\mathscr{A})=g\cdot E(f|\mathscr{A})$$
   $\mu$-almost everywhere.
\item $\mathscr{A}'\subseteq\mathscr{A}$ is a sub-$\sigma$-algebra, then
    $$E\big(E(f|\mathscr{A})\big|\mathscr{A}'\big)=E(f|\mathscr{A}')$$
    $\mu$-almost everywhere.
\end{enumerate}

The conditional expectation $E(\cdot|\mathscr{A})$ may be thought of as the natural projection map from $L^1(X,\mathscr{B},\mu)$ to its closed subspace $L^1(X,\mathscr{A},\mu)$. The following result is well-known (see e.g. \cite[Theorem 14.26]{Glasner}, \cite[Section 5.2]{Einsiedler-Ward}).
\begin{thm}[Martingale Theorem]\label{Martingale thm} Let $(\mathscr{A}_n)_{n\geq 1}$ be a decreasing sequence (resp. an increasing sequence) of sub-$\sigma$-algebras of $\mathscr{B}$ and let $\mathscr{A}=\bigcap\limits_{n\geq1}\mathscr{A}_n$ (resp. $\mathscr{A}=\bigvee_{n\geq 1}\mathscr{A}_n$). Then for every $f\in L^2(X,\mathscr{B},\mu)$, one has
$$E(f|\mathscr{A}_n)\to E(f|\mathscr{A})$$
in $L^2(\mu)$ and also $\mu$-almost everywhere.
\end{thm}

Let $(X,\mathscr{B},\mu)$ be a Borel probability space, and $\mathscr{A}\subseteq\mathscr{B}$ a $\sigma$-algebra. Then $\mu$ can \emph{be disintegrated over $\mathscr{A}$} as
\[
\mu=\int_X \mu_x^{\mathscr{A}} d\mu(x)
\]
in the sense that for any $f\in L^1(X,\mathscr{B},\mu)$, one has
\begin{equation}
E(f|\mathscr{A})(x)=\int f(y)d\mu_x^{\mathscr{A}}(y)\quad
\text{for }\mu\text{-a.e.\ }x\in X,\label{eq:Ef-mux}
\end{equation}
where $\mu_x^{\mathscr{A}}\in \mathcal{M}(X)$.
If $\mathscr{A}$ is countably-generated, then $\mu_x^{\mathscr{A}}([x]_{\mathscr{A}})=1$
for all $\mu$-almost every $x\in X$, where
$$[x]_{\mathscr{A}}=\bigcap_{x\in A,A\in\mathscr{A}}A$$
is the atom of $\mathscr{A}$ containing $x$.
Moreover $\mu_x^{\mathscr{A}}=\mu_y^{\mathscr{A}}$ for $\mu$-almost every $x,y\in X$ whenever $[x]_{\mathscr{A}}=[y]_{\mathscr{A}}$.

Let $\mu\in\mathcal{M}(X)$ and $\mu=\int_X \mu_x^{\mathscr{A}}d\mu(x)$ be the disintegration of $\mu$ over $\mathscr{A}$.
The \emph{relatively independent self-joining of $\mu$ over $\mathscr{A}$} is the probability measure
\[
\mu\times_{\mathscr{A}}\mu
=\int_X\mu_x^{\mathscr{A}}\times\mu_x^{\mathscr{A}}d\mu(x)
\]
 on $X\times X$ in the sense that
$$\mu\times_{\mathscr{A}}\mu(A\times B)=\int_X\mu_x^{\mathscr{A}}(A)\mu_x^{\mathscr{A}}(B)d\mu(x)$$
for all $A,B\in\mathscr{B}$. Denote $\lambda=\mu\times_{\mathscr{A}}\mu$.
Let $\pi:X\times X\to X$ be the canonical projection to the first coordinate, and
$$\lambda=\int_{X\times X}\lambda_{(x,y)}^{\pi^{-1}(\mathscr{A})}d\lambda\big((x,y)\big)$$
be the disintegration of $\lambda$ over the $\sigma$-algebra $\pi^{-1}(\mathscr{A})$ of $\mathscr{B}\times\mathscr{B}$. By \cite[Proposition 6.16]{Einsiedler-Ward}, for $\lambda$-a.e. $(x,y)$, $\lambda_{(x,y)}^{\pi^{-1}(\mathscr{A})}=\mu_x^{\mathscr{A}}\times\mu_y^{\mathscr{A}}$. Thus for $f_1,f_2\in L^2(X,\mathscr{B},\mu)$ one has
\begin{align}\label{e8}
E(f_1\otimes f_2|\pi^{-1}(\mathscr{A}))(x,y)&=\int f_1\otimes f_2(z_1,z_2)d\lambda_{(x,y)}^{\pi^{-1}(\mathscr{A})}(z_1,z_2)\nonumber\\
&=\int f_1(z_1)\cdot f_2(z_2)d\mu_x^{\mathscr{A}}\times\mu_y^{\mathscr{A}}(z_1,z_2)\nonumber\\
&=\int f_1(z_1)d\mu_x^{\mathscr{A}}(z_1)\cdot\int f_2(z_2)d\mu_y^{\mathscr{A}}(z_2)\nonumber\\
&=E(f_1|\mathscr{A})(x)\cdot E(f_2|\mathscr{A})(y),
\end{align}
where $f_1\otimes f_2$ is the function on $X\times X$
defined by
$f_1\otimes f_2(x_1,x_2)=f_1(x_1)f_2(x_2)$ for $(x_1,x_2)\in X\times X$.

\section{$\Delta$-weakly mixing subsets along a collection of sequences}
In this section, we study properties of $\Delta$-transitive subsets
and $\Delta$-weakly mixing subsets along a collection of sequence.
As those sets can be regarded as subsequence version of
$\Delta$-transitive subsets and
$\Delta$-weakly mixing subsets.
The idea of the proofs are the same, we only state the results and outline the key ingredients of the proofs. We will leave details of the proof to the interested reader.

Let $\Lambda=\{a_1,a_2,\dotsc, a_{\ell}\}$ be a collection of sequences and $(X,T)$ a topological dynamical system.
Recall that a closed subset $E$ of $X$ is called \emph{$\Delta$-transitive along the collection $\Lambda$ of sequences}
if there exists a residual subset $A$ of $E$ such that for every $x\in A$ one has
\[ E^\ell  \subseteq \overline{\{ (T^{a_1(n)}x,\dotsc,T^{a_\ell(n)}x)\colon n\in\bbn\}}.\]
Following the idea in \cite{HLYZ17}, we define the hitting time set
of subsets  along the collection $\Lambda$ of sequences.
For subsets $V$, $U_1,U_2,\dotsc,U_\ell$  of $X$, define
\[
N_{\Lambda}(V;U_1,U_2,\dotsc,U_{\ell})=
\biggl\{k\in\bbn: V\cap\bigcap_{i=1}^{\ell}T^{-a_i(k)}U_i\neq\emptyset\biggr\}.
\]

We have the following characterization of $\Delta$-transitive sets along a collection of sequences.
The proof is similar to the one of \cite[Proposition 3.3]{HLYZ17}.

\begin{prop}\label{prop:Delta-transitive-Lambda}
Let $\Lambda=\{a_1,a_2,\dotsc, a_{\ell}\}$ be a collection of sequences and $(X,T)$ a topological dynamical system.
Then a closed subset $E$ of $ X$ is $\Delta$-transitive
along $\Lambda$ if and only if
for every non-empty open subsets $V$, $U_1,U_2,\dotsc,U_\ell$  of $X$
intersecting $E$,
\[
N_{\Lambda}(V\cap E;U_1,U_2,\dotsc,U_{\ell})\neq\emptyset.
\]
\end{prop}

Recall that $E$ is called \emph{$\Delta$-weakly mixing along the collection $\Lambda$ of sequences} if $E^m$ is a $\Delta$-transitive subset of $(X^m,T^{(m)})$ along $\Lambda$ for every $m\in\bbn$.
By Proposition~\ref{prop:Delta-transitive-Lambda},
we have the following characterization of $\Delta$-weakly mixing sets along a collection of sequences.

\begin{prop}\label{prop1}
Let $\Lambda=\{a_1,a_2,\dotsc,a_{\ell}\}$ be a collection of sequences and $(X,T)$ a topological dynamical system.
Then a closed subset $E$ of $ X$ is $\Delta$-weakly mixing
along $\Lambda$ if and only if for every $m\in\bbn$ and non-empty open subsets $V_j$, $U_{i,j}$, $i\in\{1,2,\dotsc,\ell\}$, $j\in\{1,2,\dotsc,m\}$,  of $X$ intersecting $E$, one has
\[\bigcap_{j=1}^{m}N_{\Lambda}(V_j\cap E;U_{1,j},U_{2,j}\dotsc,U_{\ell,j})\neq\emptyset.
\]
\end{prop}

It is not hard to see that if a $\Delta$-weakly mixing set along a collection of sequences with $\ell\geq 2$ has at least two points then it  must be perfect.

Let $E$ be a closed subset of $X$.
For $\varepsilon>0$,
we say that a subset $A$ of $X$ is \emph{$(\Lambda,\varepsilon)$-spread in $E$}
if there exist $\delta\in(0,\varepsilon)$, $n\in\bbn$ and distinct points $z_1,z_2,\dotsc,z_n\in X$
such that $A\subset \bigcup_{i=1}^n B(z_i,\delta)$ and
for any maps $g_j\colon\{z_1,z_2,\dotsc,z_n\}\to E$ where $j=1,2,\dotsc,\ell$,
there exists $k\in\bbn$ such that $\frac{1}{k}<\varepsilon$
and
$T^{a_j(k)}(B(z_i,\delta))\subset B(g_j(z_i),\varepsilon)$ for $i=1,2,\dotsc,n$ and $j=1,2,\dotsc,\ell$.
Denote by $\mathscr{X}(\Lambda,\varepsilon,E)$ the collection of all closed sets that are $(\Lambda,\varepsilon)$-spread in $E$ and put
\[
\mathscr{X}(\Lambda,E)=\bigcap_{k=1}^\infty \mathscr{X}(\Lambda,\tfrac{1}{k},E).
\]
Similar to \cite[Lemma 3.8.]{HLYZ17}, we can show the following result.
\begin{lem}\label{lem:spread-sets}
If $E$ is a $\Delta$-weakly mixing set along the collection $\Lambda$ of sequence,
then $\mathscr{X}(\Lambda,E)\cap 2^E$ is residual in $2^E$,
where $2^E$ is the hyperspace of $E$ with the Hausdorff metric.
\end{lem}

Now combing Lemma \ref{lem:spread-sets} and the method
in the proof of \cite[Theorem A]{HLYZ17}, we have the following chaotic behavior of $\Delta$-weakly mixing sets along a collection of sequences.

\begin{thm}\label{thm-Delta-weak-mixing}
	Let $(X,T)$ be a topological dynamical system and $\Lambda=\{a_1,\dotsc,a_\ell\}$ be a collection of sequences of integers. Then a closed subset $E$ of $X$ with at least two points is $\Delta$-weakly mixing along $\Lambda$ if and only if there exists an increasing sequence of Cantor sets $C_1\subset C_2\subset \dotsb $ of $E$ such that $C:=\bigcup_{k=1}^\infty C_k$ is dense in $E$ and
	\begin{enumerate}
		\item for any subset $A$ of $C$ and any continuous functions
		$g_j\colon A\to E$ for $j=1,\dotsc,\ell$, there exists
		an increasing sequence $\{p_n\}$ of positive integers
		such that
		\[\lim_{n\to\infty } T^{a_j(p_n)}x=g_j(x),\]
		for every $x\in A$ and $j=1,\dotsc,\ell$;
		\item for any $k\in\bbn$, any closed subset $B$ of $C_k$,
		and continuous function $h_j\colon B\to E$ for $j=1,\dotsc,\ell$,
		there exists an increasing sequence $\{q_n\}$ of positive integers
		such that
		\[\lim_{n\to\infty } T^{a_j(q_n)}x=h_j(x),\]
		uniformly on $x\in B$ and $j=1,\dotsc,\ell$.
	\end{enumerate}
\end{thm}

Let $\Lambda=\{a_1,a_2,\dotsc,a_\ell\}$ be a collection of sequences, and $(X,T)$ be a topological dynamical system with a metric $\rho$ on $X$.
For $n\in\bbn$, we say that an $n$-tuple $(x_1,x_2,\dotsc,x_n)\in X^n$ is \emph{$\Delta$-Li-Yorke chaotic along $\Lambda$}, if
\[
\limsup\limits_{k\to\infty}\min\limits_{1\leq i<j\leq n;1\leq p,q\leq \ell}\rho(T^{a_p(k)}x_i,T^{a_q(k)}x_j)>0
\]
and
\[
\liminf\limits_{k\to\infty}\max\limits_{1\leq i<j\leq n;1\leq p,q\leq \ell}\rho(T^{a_p(k)}x_i,T^{a_q(k)}x_j)=0,
\]
and \emph{$\Delta^*$-Li-Yorke chaotic along $\Lambda$}, if
\[
\limsup\limits_{k\to\infty}\min\limits_{1\leq i,j\leq n;1\leq p<q\leq \ell}\rho(T^{a_p(k)}x_i,T^{a_q(k)}x_j)>0
\]
and
\[
\liminf\limits_{k\to\infty}\max\limits_{1\leq i,j\leq n;1\leq p<q\leq \ell}\rho(T^{a_p(k)}x_i,T^{a_q(k)}x_j)=0.
\]
A subset $K$ of $X$ is called \emph{$\Delta$-Li-Yorke $n$-chaotic along $\Lambda$} (resp. \emph{$\Delta^*$-Li-Yorke $n$-chaotic along $\Lambda$})
if for any pairwise distinct points $x_1,x_2,\dotsc,x_n\in K$,
the $n$-tuple $(x_1,x_2,\dotsc,x_n)$ is $\Delta$-Li-Yorke chaotic along $\Lambda$ (resp. $\Delta^*$-Li-Yorke chaotic along $\Lambda$).
The topological dynamical system $(X,T)$ is called \emph{$\Delta$-Li-Yorke $n$-chaotic along $\Lambda$} (resp. \emph{$\Delta^*$-Li-Yorke $n$-chaotic along $\Lambda$})
if there is an uncountable  $\Delta$-Li-Yorke $n$-chaotic set along $\Lambda$ (resp. $\Delta^*$-Li-Yorke $n$-chaotic set along $\Lambda$).

\begin{rem}
If $\Lambda$ is only composed of a sequence $a(n)$,
then a $\Delta$-Li-Yorke $n$-chaotic tuple along $\Lambda$
is just the Li-Yorke $n$-chaotic tuple along the sequence $a(n)$.
If $\Lambda$ contains at least two sequences,
then a $\Delta^*$-Li-Yorke chaotic tuple along $\Lambda$
can be regarded as a multi-variant version of asynchronous Li-Yorke chaotic tuple, which was introduced in \cite{HLYZ17}.
\end{rem}

\begin{prop}\label{Li-Yorke chaotic}
Let $\Lambda=\{a_1,a_2,\dotsc,a_{\ell}\}$ be a collection of sequences
and $(X,T)$ be a topological dynamical system.
If there exists a $\Delta$-weakly mixing set along $\Lambda$ with at least two points,
then for every $n\geq 2$, $(X,T)$ is  $\Delta$-Li-Yorke $n$-chaotic
and  $\Delta^*$-Li-Yorke $n$-chaotic along $\Lambda$.
\end{prop}
\begin{proof}
Let $E$ be a  $\Delta$-weakly mixing set along $\Lambda$ with at least two points.
By Theorem \ref{thm-Delta-weak-mixing} there exists an increasing sequence of Cantor subsets $C_1\subset C_2\subset\dotsb$ of $E$, such that $C:=\bigcup\limits_{i=1}^{\infty}C_i$ is dense in $E$ and has the properties introduced in the Theorem \ref{thm-Delta-weak-mixing}.
For any $n\in\bbn$, now we shall show that $C$ is $\Delta$-Li-Yorke $n$-chaotic and $\Delta^*$-Li-Yorke $n$-chaotic along $\Lambda$.

Fix any distinct $x_1,x_2,\dotsc,x_n\in C$.
Since $E$ is $\Delta$-weakly mixing along $\Lambda$ with at least two points, $E$ is perfect. Then exist pairwise distinct points $e_1,e_2,\dotsc,e_{n\ell}\in E$. Let
$$\delta=\frac{1}{2}\min\limits_{1\leq i<j\leq n\ell}\rho(e_i,e_j)>0.$$
First we set $h_q:\{x_1,x_2,\dotsc,x_n\}\to E$, $h_q(x_i)=e_1$, for $i=1,2,\dotsc,n$ and $q=1,2,\dotsc,\ell$. Then there exists an increasing sequence $\{q_k\}$ of positive integers such that $\lim\limits_{k\to\infty}T^{a_q(q_k)}x_i=h_q(x_i)=e_1$, for every $i=1,2,\dotsc,n$ and $q=1,2,\dotsc\ell$. Thus we have
$$\lim\limits_{k\to\infty}\max\limits_{1\leq i<j\leq n;1\leq p,q\leq\ell}\rho(T^{a_p(q_k)}x_i,T^{a_q(q_k)}x_j)=\rho(e_1,e_1)=0,$$
and
$$\lim\limits_{k\to\infty}\max\limits_{1\leq i,j\leq n;1\leq p<q\leq\ell}\rho(T^{a_p(q_k)}x_i,T^{a_q(q_k)}x_j)=\rho(e_1,e_1)=0.$$
And we set $h_q':\{x_1,x_2,\dotsc,x_n\}\to E$, $h_q'(x_i)=e_{(i-1)\ell+q}$. Then there exists an increasing sequence $\{q_k'\}$ of positive integers such that $\lim\limits_{k\to\infty}T^{a_j(q_k)}x_i=h_q'(x_i)$ for every $i=1,2,\dotsc,n$ and $q=1,2,\dotsc,\ell$. Thus we have
\[\lim\limits_{k\to\infty}\min\limits_{1\leq i<j\leq n;1\leq p,q\leq\ell}\rho(T^{a_p(q_k)}x_i,T^{a_q(q_k)}x_j)
=\min\limits_{1\leq i<j\leq n;1\leq p,q\leq\ell}\rho(h_p'(x_i),h_q'(x_j))>\delta.\]
and
\[\lim\limits_{k\to\infty}\min\limits_{1\leq i,j\leq n;1\leq p<q\leq\ell}\rho(T^{a_p(q_k)}x_i,T^{a_q(q_k)}x_j)\\
=\min\limits_{1\leq i,j\leq n;1\leq p< q\leq\ell}\rho(h_p'(x_i),h_q'(x_j))>\delta.\]
This ends the proof.
\end{proof}

\section{Proof of theorem \ref{main-thm1}}
The aim of this section is to prove Theorem \ref{main-thm1}.
To do this, we first   recall some notions and basic results of entropy of a measure preserving system.

Let $(X,\mathscr{B},\mu,T)$ be a measure preserving system.
For a finite measurable partition $\alpha$,
the measure-theoretic entropy of $\mu$ relative to $\alpha$,
denoted by $h_{\mu}(T,\alpha)$,
is defined as
$$h_{\mu}(T,\alpha)=
\lim_{n\to\infty}\frac{1}{n}H_{\mu}
\biggl(\bigvee_{i=0}^{n-1}T^{-i}\alpha\biggr),$$
where $H_{\mu}(\alpha)=-\sum\limits_{A\in\alpha}\mu(A)\log \mu(A)$.
The measure-theoretic entropy of $\mu$ is defined as
$$h_{\mu}(X,T)=\sup\limits_{\alpha}h_{\mu}(T,\alpha),$$
where the supremum ranges over all finite partitions of $X$.

The Pinsker $\sigma$-algebra of a system $(X,\mathscr{B},\mu,T)$ is defined as
$$P_{\mu}(T)=\{A\in\mathscr{B}: h_{\mu}(T,\{A,X\backslash A\})=0\}.$$
It is easy to see that $P_{\mu}(T)$ is $T$-invariant. The Rohlin-Sinai theorem identifies the Pinsker $\sigma$-algebra as the ``remote past" of a generating partition (see \cite{Rokhlin-Sinai}).

\begin{thm}\label{Rohlin-Sinai} Let $(X,\mathscr{B},\mu,T)$  be a measure preserving system and $P_{\mu}(T)$ be its Pinsker sub-$\sigma$-algebra.  Then there exists  a sub-$\sigma$-algebra $\mathscr{P}$ of $\mathscr{B}$ such that
\[T^{-1}\mathscr{P}\subset \mathscr{P},\ \bigvee_{k=1}^{\infty}T^k\mathscr{P}=\mathscr{B}\text{ and } \bigcap_{k=0}^{\infty}T^{-k}\mathscr{P}=P_{\mu}(T).\]
\end{thm}

Now we are going to prove Theorem~\ref{main-thm1}.
Note that we follow some ideas in \cite{Assain}, \cite{Derrien-Lesigne} and \cite{Li-Qiao}.
\begin{proof}[Proof of Theorem \ref{main-thm1}]
Let $(X,\mathscr{B},\mu,T)$  be a measure preserving system.	
Choose a sub-$\sigma$-algebra $\mathscr{P}$ as in the Theorem \ref{Rohlin-Sinai}. Firstly, we assume $\mathscr{P}$-measurable.
Since $f_i\in L^{\infty}$, without lost of generality, we can assume $\Vert f_i\Vert _{\infty}\leq 1$, $i=1,2,\dotsc,\ell$. Given $\varepsilon>0$, by Theorem \ref{Martingale thm} there exists $m\in\bbn$ such that
$$
\Vert E(f_i|P_{\mu}(T))-E(f_i|T^{-m}\mathscr{P})\Vert_{L^2(\mu)}
<\varepsilon
$$
for every $i=1,2,\dotsc,\ell$.
To keep notations simply,  let $f_i^{\infty}=E(f_i|P_{\mu}(T))$ and $f_i^{m}=E(f_i|T^{-m}\mathscr{P})$ for $i=1,2,\dotsc,\ell$.
We also let $f_0=f_{\ell+1}=1$.
Then for every $N\in\bbn$ one has:

\begin{align}\label{e0}
&\biggl\Vert \frac{1}{N}\sum_{n=1}^{N}\biggl(\prod_{i=1}^{\ell}T^{a_i(n)}f_i-\prod_{i=1}^{\ell}T^{a_i(n)}f_i^{\infty}\biggr) \biggr\Vert _{L^2(\mu)}\nonumber\\
&=\biggl\Vert \frac{1}{N}\sum_{n=1}^{N}\biggl(\prod_{i=1}^{\ell}T^{a_i(n)}f_i^{\infty}-T^{a_1(n)}f_1\cdot\prod_{i=2}^{\ell}T^{a_i(n)}f_i^{\infty}+
T^{a_1(n)}f_1\cdot\prod_{i=2}^{\ell}T^{a_i(n)}f_i^{\infty}-\dotsc\nonumber\\
&\qquad\qquad-\prod_{i=1}^{\ell-1}T^{a_i(n)}f_i\cdot T^{a_{\ell}(n)}f_{\ell}^{\infty}+\prod_{i=1}^{\ell-1}T^{a_i(n)}f_i\cdot T^{a_{\ell}(n)}f_{\ell}^{\infty}-\prod_{i=1}^{\ell}T^{a_i(n)}f_i\biggr)\biggr\Vert_{L^2(\mu)}\nonumber\\
&\leq\sum_{h=1}^{\ell}\biggl\Vert\frac{1}{N}\sum_{n=1}^{N}\big(T^{a_1(n)}f_1\dotsc T^{a_h(n)}(f_h-f_h^{\infty}) T^{a_{h+1}(n)}f_{h+1}^{\infty}\dotsc T^{a_\ell(n)}f_\ell^{\infty}
\big) \biggr\Vert_{L^2(\mu)}\nonumber\\
&\leq\sum_{h=1}^{\ell}\biggl\Vert\frac{1}{N}\sum_{n=1}^{N}\big(T^{a_1(n)}f_1\dotsc T^{a_h(n)}(f_h-f_h^{m}) T^{a_{h+1}(n)}f_{h+1}^{\infty}\dotsc T^{a_\ell(n)}f_\ell^{\infty}
\big) \biggr\Vert_{L^2(\mu)}\nonumber\\
&+\sum_{h=1}^{\ell}\biggl\Vert\frac{1}{N}\sum_{n=1}^{N}\big(T^{a_1(n)}f_1\dotsc T^{a_h(n)}(f_h^{m}-f_h^{\infty}) T^{a_{h+1}(n)}f_{h+1}^{\infty}\dotsc T^{a_\ell(n)}f_\ell^{\infty}
\big) \biggr\Vert_{L^2(\mu)}\nonumber\\
&\leq \sum_{h=1}^{\ell}\biggl\Vert\frac{1}{N}\sum_{n=1}^{N}\big(T^{a_1(n)}f_1\dotsc T^{a_h(n)}(f_h-f_h^{m}) T^{a_{h+1}(n)}f_{h+1}^{\infty}\dotsc T^{a_\ell(n)}f_\ell^{\infty}
\big) \biggr\Vert_{L^2(\mu)}+ \ell\varepsilon.
\end{align}

Since $\Lambda$ satisfies Condition $(**)$, there exists $N_0\in\bbn$ such that $a_{i+1}(n)>a_i(n)+m$, for every $n>N_0$ and $i=1,2,\dotsc,\ell-1$. And there exists $N_1\in\bbn$ large enough, such that $\frac{N_0}{N_1}<\frac{\varepsilon^2}{16}$ and for every $i=1,2,\dotsc,\ell$, one has
$$\frac{1}{N_1^2}\#\{(n,k)\in [1,N_1]^2: \  |a_i(n)-a_i(k)|\leq m\}<\frac{\varepsilon^2}{8}.$$

For every fixed $i=1,2,\dotsc,\ell$, let $g_n^i=T^{a_1(n)}f_1^{\infty}\cdot T^{a_2(n)}f_2^{\infty}\dotsc T^{a_{i-1}(n)}f_{i-1}^{\infty}$, and $h_n^i=T^{a_{i+1}(n)}f_{i+1}\cdot T^{a_{i+2}(n)}f_{i+2}\dotsc T^{a_\ell(n)}f_\ell$, one has
\begin{align}\label{e1}
&\biggl\Vert\frac{1}{N_1}\sum_{n=1}^{N_1}(g_n^i\cdot T^{a_i(n)}f_i\cdot h_n^i-g_n^i\cdot T^{a_i(n)}f_i^{m}\cdot h_n^i)\biggr\Vert_{L^2({\mu})}^2\nonumber\\
&=\frac{1}{N_1^2}\sum_{n,k=1}^{N_1}\int \big(g_n^i T^{a_i(n)}f_i h_n^i-g_n^i T^{a_i(n)}f_i^m h_n^i\big)\big(g_k^i T^{a_i(k)}f_i h_k^i-g_k^i T^{a_i(k)}f_i^mh_k^i\big)d\mu\nonumber\\
&=\frac{1}{N_1^2}\sum_{n,k=1}^{N_1}(A_{nk}^i-B_{nk}^i+C_{nk}^i-D_{nk}^i),
\end{align}
where
$$A_{nk}^i=\int g_n^i\cdot g_k^i\cdot T^{a_i(n)}f_i\cdot T^{a_i(k)}f_i\cdot h_n^i\cdot h_k^i d\mu,$$
$$B_{nk}^i=\int g_n^i\cdot g_k^i\cdot T^{a_i(n)}f_i\cdot T^{a_i(k)}f_i^m\cdot h_n^i\cdot h_k^i d\mu,$$
$$C_{nk}^i=\int g_n^i\cdot g_k^i\cdot T^{a_i(n)}f_i^m\cdot T^{a_i(k)}f_i^m\cdot h_n^i\cdot h_k^id\mu,$$
$$D_{nk}^i=\int g_n^i\cdot g_k^i\cdot T^{a_i(n)}f_i^m\cdot T^{a_i(k)}f_i\cdot h_n^i\cdot h_k^id\mu.$$
For given $(n,k)\in\{N_0+1,N_0+2,\dotsc,N_1\}^2$, there are three cases.

Case 1: $a_i(n)>a_i(k)+m$. Recall that, $g_n^i,\ g_k^i$ are $P_{\mu}(T)$-measurable and $f_1,f_2,\dotsc,f_\ell$ are $\mathscr{P}$-measurable function, hence $g_n^i\cdot g_k^i\cdot T^{a_i(n)}f_i\cdot h_n^i\cdot h_k^i$ is $T^{-(a_i(k)+m)}\mathscr{P}$-measurable. Then one has
\begin{align*}
&A_{nk}^i=\int E\bigg(g_n^i\cdot g_k^i\cdot T^{a_i(n)}f_i\cdot T^{a_i(k)}f_i\cdot h_n^i\cdot h_k^i\big|T^{-(a_i(k)+m)}\mathscr{P}\bigg)d\mu\\
&=\int g_n^i\cdot g_k^i\cdot T^{a_i(n)}f_i \cdot E\bigg(T^{a_i(k)}f_i \big|T^{-(a_i(k)+m)}\mathscr{P}\bigg)\cdot h_n^i\cdot h_k^i d\mu\\
&=\int g_n^i\cdot g_k^i\cdot T^{a_i(n)}f_i \cdot T^{a_i(k)}E\bigg(f_i \big|T^{-m}\mathscr{P}\bigg)\cdot h_n^i\cdot h_k^i d\mu\\
&=B_{nk}^i.
\end{align*}
And
\begin{align*}
&D_{nk}^i=\int E\bigg(g_n^i\cdot g_k^i\cdot T^{a_i(n)}f_i^m\cdot T^{a_i(k)}f_i\cdot h_n^i\cdot h_k^i\big|T^{-(a_i(k)+m)}\mathscr{P}\bigg)d\mu\\
&=\int g_n^i\cdot g_k^i\cdot T^{a_i(n)}f_i^m\cdot E\bigg(T^{a_i(k)}f_i|T^{-(a_i(k)+m)}\mathscr{P}\bigg)\cdot h_n^i\cdot  h_k^id\mu \\
&=\int g_n^i\cdot g_k^i\cdot T^{a_i(n)}f_i^m\cdot T^{a_i(k)}E\bigg(f_i|T^{-m}\mathscr{P}\bigg)\cdot h_n^i\cdot  h_k^id\mu \\
&=C_{nk}^i.
\end{align*}

Case 2: $a_i(k)>a_i(n)+m$. Similar with Case 1 we have $A_{nk}^i=D_{nk}^i$ and $B_{nk}^i=C_{nk}^i$

Case 3: $|a_i(k)-a_i(n)|\leq m$. Since $\Vert f_i\Vert_{\infty}\leq 1$, one has
$$\int \bigg| \big(g_n^i\cdot T^{a_i(n)}f_i\cdot h_n^i-g_n^i\cdot T^{a_i(n)}f_i^m\cdot h_n^i\big)\cdot\big(g_k^i\cdot T^{a_i(k)}f_i\cdot h_k^i-g_k^i\cdot T^{a_i(k)}f_i^m\cdot h_k^i\big)\bigg|d\mu\leq 4.$$

To summing up, we set $F=\{1,2,\dotsc,N_0\}\times\{1,2,\dotsc,N_1\}\cup\{1,2,\dotsc,N_1\}\times\{1,2,\dotsc,N_0\}$,
\begin{align*}
(\ref{e1})&=\frac{1}{N_1^2}\bigg(\sum_{(n,k)\in F}\int\big(g_n^i\cdot T^{a_i(n)}(f_i-f_i^m)\cdot h_n^i\big)\cdot\big(g_k^i\cdot T^{a_i(k)}(f_i-f_i^m)\cdot h_k^i\big)d\mu\\
&+\sum_{\mbox{\tiny$\begin{array}{c}
N_0<n,k\leq N_1,\\
|a_i(n)-a_i(k)|\leq m\end{array}$}}\int\big(g_n^i\cdot T^{a_i(n)}(f_i-f_i^m)\cdot h_n^i\big)\cdot\big(g_k^i\cdot T^{a_i(k)}(f_i-f_i^m)\cdot h_k^i\big)d\mu\bigg)\\
&\leq 4\biggl(\frac{\#\{(n,k)\in [1,N_1]^2: \ |a_i(n)-a_i(k)|\leq m\}}{N_1^2}+2\frac{N_0}{N_1}\biggr)<4\Bigl(\frac{\varepsilon^2}{8}+2\frac{\varepsilon^2}{16}\Bigr)=\varepsilon^2.
\end{align*}

Thus we have ($\ref{e0})\leq \ell\varepsilon+\ell\varepsilon=2\ell\varepsilon$. Thus the conclusion holds for
$\mathscr{P}$-measurable functions in $L^{\infty}(\mu)$ because of the arbitrary of $\varepsilon$, also for $T^r\mathscr{P}$-measurable functions, $r\in\bbn$, for $\mu$ is $T$-invariant. For general functions $f_1,f_2,\dotsc,f_{\ell}\in L^{\infty}(\mu)$, there exists $T^k\mathscr{P}$-measurable functions $f_{1k},f_{2k},\dotsc,f_{\ell k}\in L^{\infty}(\mu)$ which  satisfy the conclusion and converge to $f_1,f_2,\dotsc,f_{\ell}$ when $k\to\infty$, and we can know this conclusion holds for $f_1,f_2,\dotsc,f_{\ell}$. Thus the result holds for all functions in $L^\infty(\mu)$.
\end{proof}

\begin{thm}\label{thm:Pinsker-simga-algebra}
	Let $\Lambda=\{a_1,a_2,\dotsc,a_\ell\}$ be a collection of sequences that satisfies Condition $(**)$.
	Let $(X,\mathscr{B},\mu,T)$ be a  measure preserving system and
	$P_\mu(T)$ be the Pinsker $\sigma$-algebra.
	Then for any $f_0,f_1,\dotsc,f_\ell\in L^\infty(\mu)$,
	\[
	\frac{1}{N}\sum_{n=1}^N\int
	\bigg| E\biggl(f_0\prod_{i=1}^{\ell}T^{a_i(n)}f_i\big|P_{\mu}(T)\biggr)
	-E\big(f_0|P_{\mu}(T)\big)
	\prod_{i=1}^{\ell}T^{a_i(n)}E\big(f_i|P_{\mu}(T)\big)\bigg|^2d\mu\to 0\]
	as $N\to \infty $.
\end{thm}
\begin{proof}
Fix $f_0,f_1,\dotsc,f_\ell\in L^\infty(\mu)$.
By Theorem \ref{main-thm1}, we have
\begin{equation}
\lim\limits_{N\to\infty}\frac{1}{N}\sum_{n=1}^{N}\biggl(\prod_{i=1}^{\ell}T^{a_i(n)}f_i-\prod_{i=1}^{\ell}T^{a_i(n)}
E(f_i|P_\mu(T))\biggr)=0 \label{eq:f-i-L2}
\end{equation}
in $L^2(\mu)$. At first, we assume that there exists some $h\in\{1,2,\dotsc,\ell\}$ such that $E(f_h|P_{\mu}(T))=0$. Then by \eqref{eq:f-i-L2}, we have
\begin{align}\label{eq:fi-mean-0}
\bigg\vert\frac{1}{N}\sum_{n=1}^{N}&\int
E\biggl(f_0\prod_{i=1}^\ell T^{a_i(n)}f_i|P_\mu(T)\biggr) d\mu\bigg\vert \nonumber \\
& =\bigg\vert\frac{1}{N}\sum_{n=1}^{N}\int
f_0\prod_{i=1}^\ell T^{a_i(n)}f_i d\mu\bigg\vert\nonumber\\
&\leq\Vert f_0\Vert_{L^2(\mu)}\cdot\Vert\frac{1}{N}\sum_{n=1}^N\prod_{i=1}^{\ell}T^{a_i(n)}f_i\Vert_{L^2(\mu)}\to0,\ as\ N\to\infty.
\end{align}

Consider the relative product space $(X\times X,\mathscr{B}\times\mathscr{B},\lambda,T\times T)$, where $\lambda=\mu\times_{P_{\mu}(T)}\mu$.
Then one has $P_{\lambda}(T)=\pi^{-1}(P_{\mu}(T))$
(see e.g. \cite[Theorem 0.4(iii)]{Danilenko}, or \cite[Theorem 4]{Glasner-Thouvenot-Weiss} for free action).
For every $g_1,g_2\in L^2(X,\mathscr{B},\mu,T)$,
by \eqref{e8}, one has
$E(g_1\otimes g_2|P_\lambda(T))=E(g_1|P_{\mu}(T))E(g_2|P_{\mu}(T))$.
Applying \eqref{eq:fi-mean-0} to the relative product space $(X\times X,\mathscr{B}\times\mathscr{B},\lambda,T)$, we have
\begin{align}
0=\lim_{N\to\infty}&\frac{1}{N}\sum_{n=1}^{N}
\int E\bigg( f_0\otimes f_0\prod_{i=1}^{\ell}T^{a_i(n)}\times T^{a_i(n)}
(f_i\otimes f_i) \big| P_{\lambda}(T)\bigg) d\lambda \nonumber\\
&=\lim_{N\to\infty}\frac{1}{N}
\sum_{n=1}^{N}\int
E\bigg(\big(f_0\prod_{i=1}^\ell T^{a_i(n)}f_i\big)\otimes
\big(f_0\prod_{i=1}^\ell T^{a_i(n)}f_i \big)
\big| P_{\lambda}(T)\bigg)d\lambda \nonumber\\
&=\lim_{N\to\infty}\frac{1}{N}\sum_{n=1}^{N}
\int E\bigg(f_0\prod\limits_{i=1}^{\ell} T^{a_i(n)}f_i\big|P_{\mu}(T)\bigg)\otimes
E\bigg(f_0\prod\limits_{i=1}^{\ell} T^{a_i(n)}f_i\big|P_{\mu}(T)\bigg)d\lambda \nonumber\\
&=\lim_{N\to\infty}\frac{1}{N}\sum_{n=1}^{N}
\int \bigg|E\bigg(f_0\prod_{i=1}^{\ell}T^{a_i(n)}f_i
\big|P_{\mu}(T)\bigg)\bigg|^2d\mu.
\label{eq:fi-P-mu-mean-0}
\end{align}
This implies the conclusion holds if $E(f_h|P_{\mu}(T))=0$, for some $h\in\{1,2,\dotsc,\ell\}$. For the general case, one has
\begin{align}\label{eq:decompose-fi}
&E\biggl(f_0\prod_{i=1}^{\ell}T^{a_i(n)}f_i\bigg|P_{\mu}(T)\biggr)
-E\big(f_0|P_{\mu}(T)\big)
\prod_{i=1}^{\ell}T^{a_i(n)}E\big(f_i|P_{\mu}(T)\big) \nonumber \\
&= E\biggl(\prod_{i=0}^{\ell}T^{a_i(n)}f_i-\prod_{i=0}^{\ell}T^{a_i(n)}E(f_i|P_{\mu}(T)) \bigg|P_\mu(T)\biggr)\nonumber\\
&=\sum_{h=0}^{\ell}E\biggl(\prod_{i=0}^{h-1}T^{a_i(n)}f_i\cdot T^{a_h(n)}\bigl(f_h-E(f_h|P_{\mu}(T))\bigr)\prod_{j=h+1}^{\ell}E(f_j|P_{\mu}(T)) \bigg|P_{\mu}(T)\biggr)\nonumber\\
&=\sum_{h=1}^{\ell}E\biggl(f_0\cdot\prod_{i=1}^{h-1}T^{a_i(n)}f_i\cdot T^{a_h(n)}\bigl(f_h-E(f_h|P_{\mu}(T))\bigr)\prod_{j=h+1}^{\ell}E(f_j|P_{\mu}(T)) \bigg|P_{\mu}(T)\biggr)\nonumber\\
\end{align}
where $a_0(n)=0$.
Note that $E\big(f_h-E(f_h|P_{\mu}(T))\big|P_{\mu}(T)\big)=0$ for $h=1,\dotsc,\ell$.
Now the result follows from
 \eqref{eq:fi-P-mu-mean-0} and \eqref{eq:decompose-fi} immediately.
\end{proof}

\begin{prop}\label{prop:Pinsker-density}
Let $\Lambda=\{a_1,a_2,\dotsc,a_\ell\}$ be a collection of sequences that satisfies Condition $(**)$.
Let $(X,\mathscr{B},\mu,T)$ be a measure preserving system,
$P_\mu(T)$ be the Pinsker $\sigma$-algebra
and $\mu=\int \mu_z d\mu(z)$ be the disintegration of $\mu$ over $P_{\mu}(T)$.
Then for any $A_0,A_1,\dotsc,A_\ell\in\mathscr{B}$ and $\varepsilon,\delta>0$,
there exists  a subset $F\subset\bbn$ with density $1$
such that for any $n\in F$,
\[
\mu\biggl\{z\in X\colon \biggl|\mu_z \biggl(A_0\cap \bigcap_{i=1}^\ell T^{-a_i(n)}A_i\biggr)
-\mu_z(A_0) \prod_{i=1}^\ell \mu_z(T^{-a_i(n)}A_i)\biggr|<\varepsilon\biggr\}>1-\delta.
\]
\end{prop}
\begin{proof}
For $i=0,1,\dotsc,\ell$, let $f_i=\mathbf{1}_{A_i}$.
By \eqref{eq:Ef-mux}, for $\mu$-a.e. $z\in X$,
\begin{align*}
E\biggl(f_0\prod_{i=1}^{\ell}T^{a_i(n)}f_i\bigg|P_{\mu}(T)\biggr)(z)
&=E\biggl(\mathbf{1}_{A_0\cap\bigcap_{i=1}^{\ell}T^{-a_i(n)}A_i}\bigg|P_{\mu}(T)\biggr)(z)\\
&=\mu_z \biggl(A_0\cap \bigcap_{i=1}^\ell T^{-a_i(n)}A_i\biggr)
\end{align*}
and
\begin{align*}
E\bigl(f_0|P_{\mu}(T)\bigr)(z)
&\prod_{i=1}^{\ell}T^{a_i(n)}E\bigl(f_i|P_{\mu}(T)\bigr)(z)\\
&=E\bigl(f_0|P_{\mu}(T)\bigr)(z)
\prod_{i=1}^{\ell}E\bigl(\mathbf{1}_{T^{-a_i(n)}A_i}|P_{\mu}(T)\bigr)(z)\\
&=\mu_z(A_0)\prod_{i=1}^{\ell}\mu_z(T^{-a_i(n)}A_i).
\end{align*}
By Theorem~\ref{thm:Pinsker-simga-algebra},
\[
\lim_{N\to\infty}\frac{1}{N}\sum\limits_{n=1}^N\int
\bigg|
\mu_z \biggl(A_0\cap \bigcap_{i=1}^\ell T^{-a_i(n)}A_i\biggr)
-
\mu_z(A_0)\prod_{i=1}^{\ell}\mu_z(T^{-a_i(k)}A_i)
\biggr| d\mu(z)=0.
\]
Now by \cite[Theorem 1.20]{W82} there exists a subset $F\subset\bbn$ with density $1$ such that for any $n\in F$,
\[
\int
\bigg|
\mu_z \biggl(A_0\cap \bigcap_{i=1}^\ell T^{-a_i(n)}A_i\biggr)
-
\mu_z(A_0)\prod_{i=1}^{\ell}\mu_z(T^{-a_i(n)}A_i)
\biggr| d\mu(z)<\varepsilon \delta.
\]
This implies that for any $n\in F$,
\[
\mu\biggl\{z\in X\colon \biggl|\mu_z \biggl(A_0\cap \bigcap_{i=1}^\ell T^{-a_i(n)}A_i\biggr)
-\mu_z(A_0) \prod_{i=1}^\ell \mu_z(T^{-a_i(n)}A_i)\biggr|\geq \varepsilon\biggr\}<\delta,
\]
and then ends the proof.
\end{proof}

\section{Proof of Theorem \ref{main-thm2}}

The aim of this section is to prove Theorem \ref{main-thm2}. To do this, we will first prove the following auxiliary lemma.
\begin{lem}\label{lem:key-lemma}
Let $\Lambda=\{a_1,a_2,\dotsc,a_\ell\}$ be a collection of sequences that is good for $\lim\inf$-$\ell$-recurrence and satisfies Condition $(**)$.
Let $(X,\mathscr{B},\mu,T)$ be a measure preserving system, $P_{\mu}(T)$ be the Pinsker $\sigma$-algebra and $\mu=\int \mu_z d\mu(z)$ be the disintegration of $\mu$ over $P_{\mu}(T)$.
For every $M\in\bbn$, if $U_1,U_2,\dotsc,U_M$ are non-empty subsets in $\mathscr{B}$ such that
$$\mu\big(\{z\in X,\mu_z(U_i)>0,i=1,2,\dotsc,M\}\big)>0,$$
then there exists $L\in\bbn$ and $c>0$ such that
$$\mu\big(\{z\in X: \mu_z(U_s^L)>c: s\in\{1,2,\dotsc,M\}^{\ell+1}\}\big)>0$$
where $U_s^L=U_{s(1)}\cap T^{-a_1(L)}U_{s(2)}\cap\dotsc\cap T^{-a_\ell(L)}U_{s(\ell+1)}$ for every $s\in\{1,2,\dotsc, M\}^{\ell+1}$.
\end{lem}
\begin{proof}
For each $p\in\bbn$, let
\[
\Omega_p=\bigl\{z\in X:
\mu_z(U_i)>\tfrac{1}{p},i=1,2,\dotsc,M\bigl\}.
\]
By the assumption, we have
$\mu\bigl(\bigcup_{p=1}^\infty \Omega_p\bigr)>0$.
Then there exists $p_0\in\bbn$ such that $\mu(\Omega_{p_0})>0$.
Since $\Lambda$ is good for $\lim\inf$-$\ell$-recurrence,
there exists $c_0>0$ such that
$$
\liminf_{N\to\infty}\frac{1}{N}\sum_{n=1}^{N}
\mu\bigl(\Omega_{p_0}\cap T^{-a_1(n)}\Omega_{p_0}\cap\dotsb\cap T^{-a_\ell(n)}\Omega_{p_0}\bigr)>c_0.
$$
Let $E=\{k\in\bbn:\mu(\Omega_{p_0}\cap T^{-a_1(k)}\Omega_{p_0}\cap\dotsb\cap T^{-a_\ell(k)}\Omega_{p_0})>c_0\}$.
Then we have $\underline{D}(E)>0$.

Fix $0<\varepsilon<\frac{1}{p_0^{\ell+1}}$ and $0<\delta<\frac{c_0}{M^{\ell+1}}$. For any $s\in\{1,2,\dotsc,M\}^{\ell+1}$ and  $n\in\bbn$, let
$$
H_s^n=
\biggl\{z\in X: \biggl|\mu_z(U_s^n)-\mu_z(U_{s(1)})\prod_{i=1}^{\ell}
\mu_z(T^{-a_i(n)}U_{s(i+1)})\biggr|<\varepsilon\biggr\},
$$
and $F_s=\{n\in\bbn:\mu(H_s^n)>1-\delta\}$.
By Proposition \ref{prop:Pinsker-density}, $F_s$ has density $1$.
 Let $$F:=\bigcap\limits_{s\in\{1,2,\dotsc,M\}^{\ell+1}} F_s.$$
Then $F$ also has density $1$.
In particular, $E\cap F\neq\emptyset$.

Choose $L\in E\cap F$ and put \[H=\bigcap\limits_{s\in\{1,2,\dotsc,M\}^{\ell+1}}H_s^L.\]
Since $\mu(H_s^L)>1-\delta$,
one has $\mu(H)>1-M^{\ell+1}\delta>0$.
Let
\[
\Omega=H\cap\Omega_{p_0}\cap\bigcap\limits_{i=1}^\ell T^{-a_i(L)}\Omega_{p_0}.\]
Since $L\in E$, one has $\mu(\Omega)>c_0-M^{\ell+1}\delta>0$.
Choose $0<c<\frac{1}{p_0^{\ell+1}}-\varepsilon$.
Then for any $z\in\Omega$ and $s\in\{1,2,\dotsc,M\}^{\ell+1}$,
$$\mu_z(U_s^L)>\mu_z
(U_{s(1)})\prod_{i=1}^\ell\mu_{T^{a_i(L)}z}(U_{s(i+1)})-
\varepsilon\geq \frac{1}{p_0^{\ell+1}}-\varepsilon>c.$$
This finishes our proof.
\end{proof}

Let $(X,T)$ be a topological dynamical system.
The topological entropy of $(X,T)$ is denoted by $h_{top}(X,T)$.
For an invariant measure $\mu\in \mathcal{M}(X,T)$,
the measure-theoretic entropy of $\mu$ is denoted by $h_\mu(X,T)$.
We have the following classical  variational principle between topological entropy and measure-theoretic entropy holds (see e.g. \cite{Glasner,W82})
$$
h_{top}(X,T)=
\sup_{\mu\in\mathcal{M}(X,T)}h_{\mu}(X,T)
=\sup_{\mu\in\mathcal{M}^e(X,T)}h_{\mu}(X,T).
$$

We will also need the following result, see  \cite[Lemma 3.1]{Dou-Ye-Zhang} and \cite[Lemma 4.3]{Huang-Xu-Yi}.

\begin{thm}\label{ergodic}
Let $(X,T)$ be a dynamical system,
$\mu\in \mathcal{M}^e(X,T)$  and $P_{\mu}(T)$ be the Pinsker $\sigma$-algebra of $(X,\mathscr{B},\mu,T)$.
For $n\in\bbn$ with $n\geq 2$, let $\lambda_n=\mu\times_{P_{\mu}(T)}\mu\times_{P_{\mu}(T)} \dotsb\times_{P_{\mu}(T)} \mu$ ($n$-times)
and $\Delta_X^n=\{(x_1,\dotsc,x_n)\in X^n\colon \exists 1\leq i<j\leq n
\text{ s.t. }x_i=x_j\}$. Then
\begin{enumerate}
\item $\lambda_n\in \mathcal{M}^e(X^n,T^{(n)})$, where $T^{(n)}=T\times T\times\dotsb\times T$ ($n$-times);
\item if $h_{\mu}(X,T)>0$, then $\lambda_n(\Delta_X^n)=0$.
\end{enumerate}
\end{thm}

After the above preparation, now we are ready to prove the main result.
\begin{proof}[Proof of Theorem \ref{main-thm2}]
Since $h_{top}(X,T)>0$, by the  variational principle there exists $\mu\in \mathcal{M}^e(X,T)$ such that $h_{\mu}(X,T)>0$. Let $P_{\mu}(T)$ be the Pinsker $\sigma$-algebra of $(X,\mathscr{B},\mu,T)$, $\mu=\int\mu_zd\mu$ the disintegration of $\mu$ over $P_{\mu}(T)$. Denote $\lambda=\mu\times_{P_{\mu}(T)}\mu$, then by Theorem \ref{ergodic} $\lambda$ is ergodic measure on $X\times X$ and $\lambda(\Delta_X^2)=0$. Thus there exist non-empty closed subsets $A_1,A_2$ of $X$, with $A_1\cap A_2=\emptyset$ and $diam(A_i)<\frac{1}{2}$ for $i=1,2$,  such that
$$0<\lambda(A_1\times A_2)=\int \mu_z\times\mu_z(A_1\times A_2)d\mu(z).$$
Then there exists $c_1>0$ such that
$\mu\big(\{z\in X:\mu_z(A_i)>c_1,i=1,2\}\big)>0$. Let $\mathscr{E}_1=\{1,2\}$ and $\mathscr{E}_{k+1}=\mathscr{E}_k\times\mathscr{E}_k\times\dotsb\times\mathscr{E}_k$ ($k+1$-times), for any $k\geq 1$. By Lemma~\ref{lem:key-lemma} and induction, we can construct non-empty closed subsets $A_{\sigma}$ of $X$ for each $\sigma\in\mathscr{E}_k$, $k\in\bbn$ with the following properties:
\begin{enumerate}
  \item for any $k>1$, there exists $L_k\in\bbn$, and a non-empty closed subset $A_{\sigma}$ of $X$ for any $\sigma=(\sigma(1),\sigma(2),\dotsc,\sigma(\ell))\in \mathscr{E}_{k}$, where $\sigma(i)\in \mathscr{E}_{k-1}, i=1,2,\dotsc,\ell$, such that
$$A_\sigma\subset A_{\sigma(1)}\cap\bigcap_{i=1}^\ell T^{-a_i(L_k)}A_{\sigma(i+1)}.$$

  \item diam$( A_{\sigma}) < 2^{-k}$, for all $\sigma\in \mathscr{E}_{k}$, $k>1$.
  \item for any $k\in\bbn$, there exists $c_k>0$ such that
$$\{z\in X: \mu_z(A_{\sigma})>c_k, \text{ for all }\sigma\in \mathscr{E}_{k}\}$$
has positive measure.
\end{enumerate}

Let $A:=\bigcap\limits_{k=1}^{\infty}\bigcup_{\sigma\in\mathscr{E}_k}A_{\sigma}$.
Now we shall show that $A$ is a $\Delta$-weakly mixing subset of $(X,T)$ along the collection $\Lambda$ of sequences
using Proposition \ref{prop1}.
For any $k\in\bbn$, $A_{\sigma}$ for all $\sigma\in\mathscr{E}_k$ are pairwise disjoint because $A_1, A_2$ are disjoint and of property (1).
Thus $A$ is a Cantor set.
For any $n\in\bbn$, non-empty open subsets $V_1,V_2,\dotsc,V_n$ and $U_{i,j}$, $i\in\{1,2,\dotsc,\ell\}$, $j\in\{1,2,\dotsc,n\}$ of $X$ intersecting $A$, there exists $K\in\bbn$ large enough such that we can choose $\sigma_1,\sigma_2,\dotsc,\sigma_n$, $\sigma_{i,j}$, $i=\{1,2,\dotsc,\ell\}$, $j=\{1,2,\dotsc,n\}$ in $\mathscr{E}_K$, such that $A_{\sigma_j}\subset V_j\cap A$, $A_{\sigma_{ij}}\subset U_{ij}\cap A$ for $i=1,2,\dotsc,\ell$, $j=1,2,\dotsc,n$. For every $(\sigma_i,\sigma_{1,i},\dotsc,\sigma_{\ell, i})\in\mathscr{E}_{K+1}$,
$i\in\{1,2,\dotsc,n\}$, one has
$$\emptyset\neq A_{(\sigma_i,\sigma_{1,i},\dotsc,\sigma_{\ell, i})}\subset A_{\sigma_i}\cap \bigcap\limits_{j=1}^{\ell} T^{-a_j(L_{K+1})}A_{\sigma_{j,i}}.$$
Thus
$$
(V_i\cap A)\cap\bigcap\limits_{j=1}^\ell T^{-a_j(L_{K+1})}U_{j,i}\neq\emptyset
$$
$i=1,2,\dotsc,n$.
This means $L_{K+1}\in N_{\Lambda}(V_i\cap A;U_{1,i},U_{2,i},\dotsc,U_{\ell,i})$ for $i=1,2,\dotsc,n$. By Proposition \ref{prop1}, $A$ is a $\Delta$-weakly mixing subset of $(X,T)$ along the collection $\Lambda$ of sequences.
\end{proof}

Now by Theorem~\ref{main-thm2} and Proposition~\ref{Li-Yorke chaotic}, we have the following corollary, which shows that positive topological entropy implies multi-variant Li-Yorke chaos along polynomial times of the shift prime numbers.

\begin{cor}
If a topological dynamical system  $(X,T)$ has positive topological entropy, then for any $\ell\in\bbn$ with $\ell\geq 2$ and any polynomials $p_1(n),\dotsc,p_\ell(n)$ with rational coefficients taking integer values on the integers and $p_i(0) = 0$ for $i=1,\dotsc,\ell$,
there exists a Cantor subset $C$ of $X$ such that for every pairwise distinct points $x_1,x_2,\dotsc,x_\ell \in C$, we have
$$\limsup_{\mathbb{P}\ni k\to\infty}\min_{1\leq i<j\leq\ell}\rho(T^{p_i(k-1)}x_i,T^{p_j(k-1)}x_j)>0$$
and
$$\liminf_{\mathbb{P}\ni k\to\infty}\max_{1\leq i<j\leq\ell}\rho(T^{p_i(k-1)}x_i,T^{p_j(k-1)}x_j)=0.$$
\end{cor}

\bibliographystyle{amsplain}

\begin{thebibliography}{10}

\bibitem{Assani}
I. Assani, \textit{Multiple recurrence and almost sure convergence for weakly mixing dynamical systems}, Israel J. Math. \textbf{103} (1998), 111--124.

\bibitem{Bereelson-Leibman}
V. Bergelson, A. Leibman, \textit{Polynomial extensions of van der Waerden's and Szemer\'edi's theorems}, J. Amer. Math. Soc. \textbf{9} (1996), no. 3, 725--753.

\bibitem{Blanchard-Glansner-Kolyada-Maass}
F. Blanchard, E. Glasner, S. Kolyada, A. Maass, \textit{On Li-Yorke pairs}, J. Reine Angew. Math. \textbf{547} (2002) 51--68.

\bibitem{Blanchard-Huang}
F. Blanchard, W. Huang,
\textit{Entropy sets, weakly mixing sets and entropy capacity},
Discrete Contin. Dyn. Syst. \textbf{20} (2008), no. 2, 275--311.

\bibitem{Derrien-Lesigne}
J.M. Derrien, E. Lesigne,
\textit{A pointwise polynomial ergodic theorem for exact endomorphisms and $K$-systems},
Ann. Inst. H. Poincar\'e Probab. Statist. \textbf{32} (1996), no. 6, 765--778.

\bibitem{Danilenko}
A. Danilenko, \textit{Entropy theory from orbit point of view}, Monatsh. Math. \textbf{134} (2001), 121--141.

\bibitem{Dou-Ye-Zhang}
D, Dou, X. Ye, G. Zhang, \textit{Entropy sequences and maximal entropy sets}, Nonlinearity \textbf{19} (2006), no. 1, 53--74.

\bibitem{D14} T. Downarowicz,  \textit{Positive topological entropy implies chaos DC2}, Proc. Amer. Math. Soc. \textbf{142} (2014), no. 1, 137--149.

\bibitem{Einsiedler-Ward}
M. Einsiedler, T. Ward, \textit{Ergodic Theory with a View Towards Number Theorey}, Graduate Texts in Mathematics, vol. 259. Springer, London, 2011.

\bibitem{F16} N. Frantzikinakis,
\textit{Some open problems on multiple ergodic averages}, Bull. Hellenic Math. Soc. \textbf{60} (2016), 41--90.

\bibitem{FHK13}
N. Frantzikinakis, B. Host, B. Kra,
\textit{The polynomial multidimensional Szemer\'edi theorem along shifted primes},
Israel J. Math. \textbf{194} (2013), no. 1, 331--348.

\bibitem{F77} H. Furstenberg,
\textit{Ergodic behavior of diagonal measures and a theorem of Szemer\'edi on arithmetic progressions},
 J. Analyse Math. \textbf{31} (1977), 204--256.

\bibitem{F81} H. Furstenberg,
\textit{Recurrence in ergodic theory and combinatorial number theory}, Princeton University Press, Princeton, N.J., 1981.


\bibitem{G94} E. Glasner,
\textit{Topological ergodic decompositions and applications to products of powers of a minimal transformation}, J. Anal. Math. \textbf{64} (1994) 241--262.

\bibitem{Glasner}
E. Glasner, \textit{Ergodic Theory via Joinings}, Math. Surveys Monoger., vol. 101, American Mathematical Society, Providence, RI, 2003.

\bibitem{Glasner-Thouvenot-Weiss}
E. Glasner, J.P. Thouvenot, B. Weiss,
\textit{Entropy theory without a past},
Ergodic Theory Dynam. Systems \textbf{20} (2000), no. 5, 1355--1370.

\bibitem{HLY14}
W. Huang, J. Li, X. Ye,
\textit{Stable sets and mean Li-Yorke chaos in positive entropy systems}, J. Funct. Anal. \textbf{266} (2014), no. 6, 3377--3394.

\bibitem{HLYZ17}
W. Huang, J. Li, X. Ye, X. Zhou,
\textit{Positive topological entropy and $\Delta$-weakly mixing sets}, Adv. Math. \textbf{306} (2017), 653--683.

\bibitem{HLY18}
W. Huang, J. Li, X. Ye, \textit{Positive entropy implies chaos along any infinite sequence}, preprint, 2020,  arXiv:2006.09601.

\bibitem{HSY16}
W. Huang, S. Shao, X. Ye,
\textit{Topological correspondence of multiple ergodic averages of nilpotent group actions},  J. Anal. Math. \textbf{138} (2019), no. 2, 687--715.

\bibitem{Huang-Xu-Yi}
W. Huang, L. Xu, Y. Yi, \textit{Asymptotic pairs, stable sets and chaos in positive entropy systems}, J. Funct. Anal. \textbf{268} (2015), no.4, 824--846.

\bibitem{Huang-Ye1}
W. Haung, X. Ye, \textit{Devaney's chaos or 2-scattering implies Li-Yorke's chaos}, Topolgy Appl. \textbf{117} (2002), no. 3, 259--272.

\bibitem{HY06} W. Huang, X. Ye, \textit{A local variational relation and applications}, Israel J. Math. \textbf{151} (2006), 237--280.

\bibitem{Iwanik}
A. Iwanik, \textit{Independence and scrambled sets for chaotic mappings}, In: The mathematical heritage of C. F. Gauss, World Sci. Publ., River Edge, NJ, 1991, 372--378.

\bibitem{Kerr-li}
D. Kerr, H. Li, \textit{Independence in topological and C$*$-dynamics} Math. Ann. \textbf{338} (2007), no.4, 869--926.

\bibitem{L15} J. Li, \textit{Localization of mixing property via Furstenberg families} Discrete Contin. Dyn. Syst. \textbf{35} (2015), no. 2, 725--740.

\bibitem{Li-Qiao}
J. Li, Y. Qiao, \textit{Mean Li-Yorke chaos along some good sequences}, Monatsh. Math. \textbf{186} (2018), no. 1, 153--173.

\bibitem{LY16} J. Li, X. Ye, \textit{Recent development of chaos theory in topological dynamics} ,
Acta Math. Sin. (Engl. Ser.) \textbf{32} (2016), no. 1, 83--114.


\bibitem{L19} K. Liu, \textit{$\Delta$-weakly mixing subset in positive entropy actions of a nilpotent group}, J. Differential Equations
\textbf{267} (2019), no. 1, 525--546.

\bibitem{M10} T.K.S. Moothathu, \textit{Diagonal points having dense orbit}, Colloq. Math. \textbf{120} (2010), 127--138.

\bibitem{Piotr-Zhang(1)}
P. Oprocha, G. Zhang, \textit{On local aspects of topological weak mixing in dimension one and beyond}, Studia Math. \textbf{202} (2011), no. 3, 261-288.

\bibitem{Piotr-Zhang(2)}
P. Oprocha, G. Zhang, \textit{On local aspects of topological weak mixing, sequence entropy and chaos}, Ergodic Theory Dynam. Systems \textbf{34} (2014), no. 5, 1615--1639.

\bibitem{Piotr-Zhang(3)}
P. Oprocha, G. Zhang, \textit{Topological aspects of dynamics of pairs, tuples and sets}, Recent progress in general topology. III, 665--709, Atlantis Press, Paris, 2014.

\bibitem{Rokhlin-Sinai}
V.A. Rokhlin, Y.G. Sinai, \textit{The construction and properties of invariant measurable partitions}, Dokl. Akad. Nauk SSSR, \textbf{141} (1961), no. 5, 1038--1041.

\bibitem{W82} P. Walter, \textit{An introduction to ergodic theory}, Graduate Texts in Mathematics, 79. Springer-Verlag, New York-Berlin, 1982.

\bibitem{Wang-Zhang}
Z. Wang, G. Zhang, \textit{Chaotic behavior of group actions}, Dynamics and numbers, 299--315, Contemp. Math., 669, Amer. Math. Soc., Providence, RI, 2016.

\bibitem{Wooley-Ziegler}
T. Wooley, T. Ziegler, \textit{Multiple recurrence and convergence along the primes}, Amer. J. Math. \textbf{134} (2012), no. 6, 1705--1732.

\bibitem{Xiong-Yang}
J. Xiong, Z. Yang, \textit{Chaos caused by a topologically mixing map}, In: Dynamical systems and related topics (Nagoya, 1990), Adv. Ser. Dynam. Systems, Vol.9, 550-572, World Sci. Publ., River Edge, NJ, 1991.

\bibitem{ZWX19} S. Zhang, H. Wang, R. Xie, \textit{Localization of $\Delta$-mixing property via Furstenberg families}, Discrete Dyn. Nat. Soc. 2019, Art. ID 6759849, 5 pp.

\end{thebibliography}

\end{document}